\documentclass[12pt,reqno]{amsart}

\usepackage{fullpage}
\usepackage{amsmath,amsfonts,amssymb}
\usepackage[all]{xy}
\usepackage[english]{babel}
\usepackage[utf8x]{inputenc}
\usepackage{url}
\usepackage{ifthen}
\usepackage{booktabs}
\usepackage{multirow}
\usepackage{array}
\usepackage{paralist}
\usepackage{enumerate}
\usepackage{placeins}
\usepackage{tikz}
\usetikzlibrary{cd}

\newtheorem{constr}{Construction}
\newtheorem{thmintro}{Theorem}

 \newtheorem{thm}{Theorem}[section]
 \newtheorem{prop}[thm]{Proposition}
 \newtheorem{cor}[thm]{Corollary}
 \newtheorem{lem}[thm]{Lemma}
 
 \theoremstyle{definition}
 
 \newtheorem{defi}[thm]{Definition}
 \newtheorem{rem}[thm]{Remark}
 \newtheorem{ex}[thm]{Example}

\newcommand{\Z}{\mathbb{Z}}

\newcommand{\R}{\mathbb{R}}
\newcommand{\C}{\mathbb{C}}
\newcommand{\T}{\mathbb{T}}

\newcommand{\CP}{\mathbb{CP}}
\newcommand{\cA}{\mathcal{A}}

\newcommand{\cY}{\mathcal{Y}}

\newcommand{\del}{\partial}
\newcommand{\delbar}{\bar{\del}}

\DeclareMathOperator{\id}{id}

\DeclareMathOperator{\de}{d}

\DeclareMathOperator{\esp}{e}

\newcommand{\Oh}{\mathcal{O}}

\allowdisplaybreaks[1]

\begin{document} 

\title[]{Complex non-Kähler manifolds that are cohomologically close to, or far from, being K\"ahler}
\author{Hisashi Kasuya, Jonas Stelzig}

\subjclass[2010]{}


\begin{abstract}
We give four constructions of non-$\del\delbar$ (hence non-K\"ahler) manifolds: (1) A simply connected page-$1$-$\partial\bar\partial$-manifold (2) A simply connected $dd^c+3$-manifold (3) For any $r\geq 2$, a simply connected compact manifold with nonzero differential on the $r$-th page of the Fr\"olicher spectral sequence. (4) For any $r\geq 2$, a pluriclosed nilmanifold with nonzero differential on the $r$-th page of the Fr\"olicher spectral sequence. The latter disproves a conjecture by Popovici. A main ingredient in the first three constructions is a simple resolution construction of certain quotient singularities with control on the cohomology.
\end{abstract}

\maketitle

\section{Introduction}
For any complex manifold $X$, we have the Fr\"olicher (or Hodge-de Rham) spectral sequence, starting from Dolbeault cohomology and converging to de Rham cohomology equipped with the Hodge filtration
\[
E_1^{p,q}=H^{p,q}_{\delbar}(X)\longrightarrow (H_{dR}^k(X;\C),F).
\]
When $X$ is compact and admits a K\"ahler metric or more generally admits a compact K\"ahler modification (Fujiki's class $\mathcal C$), this spectral sequence degenerates on the first page. Moreover, in this case, the Hodge filtration $F$ and its conjugate $\bar F$ induce a pure Hodge structure on de Rham cohomology, which is succinctly characterized as:
\[
b_k^{p,q}(M):=\dim\operatorname{gr}^p_F\operatorname{gr}^q_{\bar{F}}H_{dR}^k(X;\C)=0\text{ whenever }k\neq p+q,
\]
where $\operatorname{gr}^p_F$ denotes the quotient of two successive filtration steps. If, like on K\"ahler manifolds, both of these properties ($E_1$-degeneration and purity) hold, one says $X$ is a $\del\delbar$-manifold.

It is well-known that for general compact complex manifolds, both of these properties may fail, even drastically so, \cite{CFG87, Pit89, CFG91, AT11, BR14, Ste22}. The examples witnessing this failure tend to be non-simply connected or of high dimension. We construct here simply connected examples for which the $\partial\bar\partial$-property fails in minimal and maximal ways:

\begin{constr}\label{constr: p1ddbar}
There exists a family of $3$-dimensional simply connected compact complex manifolds $M_b$, $b\in \R_{>0}$ such that, for all $b$, the Hodge filtrations induce a pure Hodge structure on $H_{dR}^k(M_b)$ for all $k$ and one has $E_2(M_b)=E_{\infty}(M_b)$, but $E_1(M_b)\neq E_2(M_b)$ if and only if $b\in\pi\Z$.
\end{constr}

The combination of pure Hodge structure and $E_2=E_\infty$ has been studied under the name page-$1$-$\del\delbar$-property or page-$2$-Hodge decomposition in \cite{PSU21, PSU22, KS23}. Most notably, any compact complex holomorphically parallelizable manifold with solvable Lie algebra of holomorphic vector fields (i.e. holomorphically parallelizable solvmanifolds) satisfies this property. All these examples are not simply connected. Since any compact complex curve and any simply connected compact complex surface satisfies the usual $\del\delbar$-property, the dimension in this construction is optimal. For $b\not\in\pi\Z$, the examples $M_b$ are three dimensional compact complex  $\del\delbar$-manifolds not in Fujiki's class $\mathcal C$ which are very different from examples given in \cite{Fri}.

\begin{constr}\label{constr: ddc+3}
There exists a simply connected compact complex $4$-fold $M$ which satisfies $E_1(M)=E_{\infty}(M)$, with
\[
b_k^{p,q}(M)=0\text{ whenever }|p+q-k|> 1
\]
but  is not a $\del\delbar$-manifold.
\end{constr}
The combination of $E_1(M)=E_{\infty}(M)$ and $b_k^{p,q}(M)=0$ for $|p+q-k|>1$ has been studied in \cite{SW22}, under the name $dd^c+3$-condition. All complex surfaces and all Vaisman manifolds are $dd^c+3$, but for those cases the simply connected examples are already K\"ahler. We do not know if there are $3$-dimensional simply connected $dd^c+3$-manifolds but have not tried too hard to find one. Note that one may construct simply connected manifolds with the same cohomological properties as in Constructions \ref{constr: p1ddbar} and \ref{constr: ddc+3} in arbitrarily high dimensions by taking products with simply connected $\del\delbar$-manifolds, e.g. $\CP^n$.

While the first two constructions give simply connected examples with `minimal' failure of the degeneration and purity properties, one may also ask for `maximal' failure. By taking products of the manifold in Construction \ref{constr: ddc+3} with itself sufficiently many times, one finds, for any $r\in\Z$ manifolds $M_r$ with $E_1(M_r)=E_\infty(M_r)$ and some $b_k^{p,q}(M_r)\neq 0$ for $|p+q-k|>r$ for any $r$, using the K\"unneth formula for the $b_k^{p,q}$ \cite[Prop. 5.8]{Ste22}. Since the degeneration step on the Fr\"olicher spectral sequence does not increase under taking products, the same argument is not possible there. Nevertheless, we show:
\begin{constr}\label{constr: FSSpi1}
For any $r$, there exists a simply connected compact complex manifold with $E_r(M_r)\neq E_\infty(M)$.
\end{constr}
This was first shown by Bigalke and Rollenske \cite{BR14} without the $\pi_1(M)=0$ hypothesis and our construction builds on theirs. 

The main idea in all of these constructions is as follows (c.f. \cite{FM08, Gua95}): Start with a non-simply connected example which admits a non-free action of a finite cyclic group $G$ such that the quotient is simply connected. Then resolve the singularities. The main technical difficulty is to keep a control over the cohomologies during this process. Since both the Fr\"olicher spectral sequence and the numbers $b_k^{p,q}$ of a complex manifold $X$ are governed by the double complex of smooth forms $(A_X,\del,\delbar)$, this amounts to keeping track of the bigraded quasi-isomorphism type of this double complex. For the first step this is standard: Differential forms on the quotient $X/G$ can be explained via (locally) invariant forms $A_{X/G}:=A_{X}^G$. The second step is more intricate. One knows how to control the double complex along blow-ups, so the most naive approach would be to blow-up submanifolds with nontrivial stabilizers of the $G$-action and hope that the quotient of the blow-up by the induced action (or a finite iteration of such steps) yields a resolution of $X/G$. Such an approach has been carried out in specific cases, see e.g. \cite{ASTT} and \cite{ST22}, but will in general already fail for $|G|=5$, see e.g. \cite{Star, Oda88}. Nevertheless, it can be extended to show the following general result:

\begin{thmintro}\label{thmintro: resolution}
Let $X$ be a compact complex manifold of dimension $n$ with an action of a finite abelian group $G$ of order $2^r\cdot 3^s$, for some $r,s\geq 0$. Then, the quotient $X/G$ admits a resolution of singularities $\tilde X\to X/G$ and a bigraded quasi-isomorphism 
\begin{equation}\label{eqn: A=AG+R}
    A_{\tilde X}\simeq A_X^G\oplus R[1,1],
\end{equation}
where $R$ looks like a bicomplex of a compact complex manifold $n-2$. 
There is an isomorphism of fundamental groups $\pi_1(X)\cong\pi_1(\tilde{X})$.
\end{thmintro}
Here, for two double complexes $A,B$, the notation $A\simeq B$ means there is a (chain of) bigraded quasi-isomorphism(s) connecting the two and $(R[1,1])^{p,q}:=R^{p-1,q-1}$. We say that a bicomplex $R$ looks like that of a compact complex manifold of dimension $k$ if $R^{p,q}=0$ unless $0\leq p,q\leq k$, and for any decomposition into indecomposable summands (see Section \ref{sec: cohm} below),
\begin{enumerate}
    \item (real structure and duality) the collection of zigzags (with multiplicity) is symmetric under mirroring along diagonal $p=q$ and antidiagonal $p+q=k$,
    \item (only dots in the corners) the only zigzags with components in degrees $(0,0)$, $(k,0)$, $(0,k)$ and $(k,k)$ are dots,
    \item (FSS degenerates in dimension $2$) if $k=2$, there are no even length zigzags.
\end{enumerate}

The structure of $R$ in Theorem \ref{thmintro: resolution} can, in principle, be explicitly described in every concrete example. 

In a somewhat different direction, one may ask whether the existence of special metrics (i.e. other than K\"ahler) implies something about the Fr\"olicher spectral sequence or the numbers $b_k^{p,q}$. For example, there are some results in this direction for Vaisman metrics (mentioned above), Astheno-K\"ahler metrics \cite{JY93, CR22} and Hermitian symplectic metrics \cite{Cav20} (the latter may or may not be all K\"ahler \cite{ST10}). It was shown in \cite{Pop19} that any compact complex manifold with an SKT (or pluriclosed) metric (i.e. $\del\delbar\omega=0$) which in addition satisfies a certian bound on the torsion satisfies $E_2=E_\infty$ and he further conjectured that this may be true for arbitrary SKT metrics. However, we show in Appendix \ref{app: SKTvsFSS}:
\begin{constr}\label{constr: FSSvsSKT}
For any $r\geq 2$, there exist compact complex nilmanifolds $M_r$ (possibly of high dimension) with an SKT metric and $E_r(M)\neq E_\infty(M)$.
\end{constr}
This construction also adapts the examples of \cite{BR14}, this time by passing to appropriately chosen torus bundles over them, which admit SKT metrics but keep the nondegeneracy properties of the Fr\"olicher spectral sequence. While the present paper was in preparation, \cite{LUV22} gave a different type of counterexample to the conjecture by pointing out that the example of \cite{Pit89}, which has $E_2\neq E_3=E_\infty$, carries an SKT metric by a result in \cite{AI03}.\\

\textbf{Acknowledgements.}
Parts of this paper were written during a stay of J.S. at the University of Osaka, and he is grateful for the invitation, financial support (provided by JSPS KAKENHI Grant Number JP19H01787) and hospitality. J.S. would also like to thank MFO (Mathematisches Forschungsinstitut Oberwolfach) and the organizers and participants of a very inspiring workshop in 2020, where the examples from Construction \ref{constr: FSSvsSKT} were found. Furthermore, J.S. thanks Daniel Greb and Jean Ruppenthal for some helpful discussions. We warmly thank Sönke Rollenske for explaining to us a specific resolution construction that evolved into what is now Theorem \ref{thmintro: resolution}. Finally, we thank J. Koll\'ar and D. Abramovich for remarks on the preprint version. In particular, D. Abramovich pointed out a gap in the proof of a previous, more general version of Theorem \ref{thmintro: resolution}.
\section{Preliminaries}

\subsection{Cohomology of compact complex manifolds} \label{sec: cohm}
A double complex (of $\C$-vector spaces) is a bigraded $\C$-vector space $A=\bigoplus_{p,q\in\Z} A^{p,q}$ together with two endomorphisms $\del$ and $\delbar$ of degree $(1,0)$ and $(0,1)$ which satisfy $d^2=0$ for $d:=\del+\delbar$ or equivalently $\del^2=\delbar^2=\del\delbar+\delbar\del=0$. We will always assume double complexes to be bounded, i.e. $A^{p,q}=0$ outside a finite region of $\Z^2$. Often, they will also come equipped with a real structure, i.e. an antilinear involution $\sigma:A\to A$ such that $\sigma A^{p,q}=A^{q,p}$ and $\sigma\del\sigma=\delbar$. The principal example is $A_X:=\cA_X(X)$, the space of smooth $\C$-valued differential forms on a complex manifold. It carries a real structure induced by conjugation on the coefficients of forms. 
By \cite{KQ20, Ste21}, any double complex can written as a direct sum of indecomposable subcomplexes and every indecomposable double complex is either a `square' or a `zigzag' (see diagrams below, where each arrow drawn is an isomorphism and all spaces and maps not drawn are trivial).

For any double complex $A$ one has the column and row filtrations
\[
F^pA:=\bigoplus_{\substack{r,s\in\Z^2\\r\geq p}}A^{r,s}\quad \bar{F}^qA:=\bigoplus_{\substack{r,s\in\Z^2\\s\geq q}}A^{r,s}
\] and these induce filtrations on total cohomology $H_d^k(A)$, denoted by the same letters. The filtrations are said to induce a \textit{pure Hodge structure} if $H^k=\bigoplus_{p+q=k}F^pH^k_{d}(A)\cap\bar F^qH^k_d(A)$. The filtrations also induce spectral sequences that converge from column, resp. row cohomology to the total cohomology, e.g.
\[
E_1^{p,q}=H^{p,q}_{\delbar}(A)\Longrightarrow (H_{d}^{p+q}(A),F).
\]
If $A=A_X$, this (column) spectral sequence is known as the \textit{Fr\"olicher spectral sequence} and the second (row) spectral sequence is determined via the real structure, so it is usually not considered. The following is a well-known tool to construct differentials on high pages of the spectral sequence, see e.g. \cite{BT82, KQ20, Ste21}.

\begin{lem}
Assume there is a direct sum decomposition of double complexes $A=B\oplus C$, with $B$ a zigzag of length $2r$, i.e. $B$ is isomorphic to the following complex of dimension $2r$
\[
\begin{tikzcd}
    \langle a_1\rangle\ar[r,"\del"]&\langle\del a_1=\delbar a_2\rangle&&&&\\
    &\langle a_2\rangle\ar[u,"\delbar"]\ar[r,"\del"]&\ddots&&&\\
    &&&\ddots&&\\
    &&&\langle a_r\rangle\ar[u,"\delbar"]\ar[r,"\del"]&\langle \del a_r\rangle.
\end{tikzcd}
\] Then, there is a nonzero differential on the $r$-th page of the Fr\"olicher spectral sequence of $A$: In fact, $d_r([a_1])=\pm [\del b_r]$.
\end{lem}
\begin{proof}
We prove this up to constant: The column cohomology of $B$ has a basis given by the two nonzero classes $[a_1]$ and $[\del a_r]$. On the other hand, the total cohomology of $B$ vanishes. Thus, there has to be a nonzero differential on $B$ as claimed. The statement for $A$ follows from the fact that spectral sequences are additive under direct sums of filtered modules.
\end{proof}
Note that because one can always take a decomposition into indecomposables, the converse is also true \cite{KQ20, Ste21}.

\begin{defi}[\cite{DGMS75}]
A complex manifold $X$ (resp. a double complex $A$) is said to have the $\del\delbar$-property, if the following equivalent conditions hold:
\begin{enumerate}
    \item The only direct summands appearing in $A_X$ (resp. $A$) are squares and zigzags of length $1$ (`dots'):
    \[
    \begin{tikzcd}
        \C\ar[r,"\sim"]&\C\\
        \C\ar[u,"\sim"]\ar[r,"\sim"]&\C\ar[u,"\sim"]
    \end{tikzcd}\quad\text{or}\quad \C.
    \]
    \item The Fr\"olicher spectral sequence(s) for $X$ (resp. $A$) degenerates on the first page and the Hodge filtrations induce a pure Hodge structure on the total cohomology.
\end{enumerate}
\end{defi}
Sometimes, the name $dd^c$-property is used instead.

\begin{defi}[\cite{PSU22}]
A complex manifold $X$ (resp. a double complex $A$) is said to have the page-$1$-$\del\delbar$ (or page-$2$-Hodge decomposition) property, if the following equivalent conditions hold:
\begin{enumerate}
    \item The only direct summands appearing in $A_X$ (resp. $A$) are squares, and zigzags of length $1$ and $2$ (`dots and lines'):
    \[
   \begin{tikzcd}
        \C\ar[r,"\sim"]&\C\\
        \C\ar[u,"\sim"]\ar[r,"\sim"]&\C\ar[u,"\sim"]
    \end{tikzcd},\quad \C,\quad \begin{tikzcd} \C\\\C\ar[u,"\sim"]\end{tikzcd},\quad\begin{tikzcd}\C\ar[r,"\sim"]&\C\end{tikzcd}.
    \]
    \item The Fr\"olicher spectral sequences for $X$ (resp. $A$) degenerate on the second page and the Hodge filtrations induce a pure Hodge structure on the total cohomology.
\end{enumerate}
\end{defi}

\begin{defi}[\cite{SW22}]
A complex manifold $X$ (resp. a double complex $A$) is said to have the $dd^c+3$-property, if the following equivalent conditions hold:
\begin{enumerate}
    \item The only direct summands appearing in $A_X$ (resp. $A$) are squares and zigzags of length $1$ and $3$ (`dots, $L$'s and reverse $L$'s'):
    \[
    \begin{tikzcd}
        \C\ar[r,"\sim"]&\C\\
        \C\ar[u,"\sim"]\ar[r,"\sim"]&\C\ar[u,"\sim"]
    \end{tikzcd},\quad \C,\quad \begin{tikzcd} \C&\\\C\ar[u,"\sim"]\ar[r,"\sim"]&\C\end{tikzcd},\quad\begin{tikzcd}\C\ar[r,"\sim"]&\C\\&\C\ar[u,"\sim"]\end{tikzcd}.
    \]
    \item The Fr\"olicher spectral sequence(s) for $X$ (resp. $A$) degenerates on the first page and the Hodge filtrations on total cohomology satisfy:
    \[
    \sum_{p+q=r}F^pH^k_d\cap \bar F^qH^k_d=\begin{cases}0 &r> k+1\\
                                        H_d^k&r<k-1.
                        \end{cases}
    \]
\end{enumerate}
\end{defi}
In terms of the numbers $b_k^{p,q}$ of the introduction, the condition on the filtrations can be written as $b_k^{p,q}=0$ as soon as $|p+q-k|>1$.
We will use:

\begin{lem}\label{lem: invariants}
Suppose $A$ is a double complex satisfying the page-$r$-$\del\delbar$-property (resp. the $dd^c+3$-property), with a finite group $G$ acting on $A$ by automorphisms of double complexes, then the double complex of invariants $A^G$ also satisfies the page-$1$-$\del\delbar$-property (resp. the  $dd^c+3$-property).
\end{lem}
\begin{proof}
By Maschke's theorem, there is an isomorphism of double complexes $A\cong \bigoplus_{\chi} A[\chi]$, where $\chi$ runs over the characters of $H$ and $A[\chi]=\{\omega\in A\mid h.\omega=\chi(h)\omega~\forall h\in H\}$ denotes the $\chi$-typical component. But a direct sum of double complexes satisfies the page-$r$-$\del\delbar$-property (resp. the  $dd^c+3$-property) if and only if all summands do.
\end{proof}We say a map between double complexes $f:A\to B$ is a bigraded quasi-isomorphism if for any decomposition $A=A^{sq}\oplus A^{zig}$ and $B=B^{sq}\oplus B^{zig}$ into squares and zigzags the induced map $f:A^{zig}\to B^{zig}$ is an isomorphism. For bounded double complexes, this is equivalent to requiring that $H_{\del}(f)$ and $H_{\delbar}(f)$ are isomorphisms, i.e. to $f$ being an $E_1$-isomorphism, see \cite{Ste21, Ste23}.
\subsection{Quotients of complex manifolds by finite groups}
For a complex manifold $X$ with an action of a finite group $G$, the quotient space $X/G$ has a canonical structure as a normal complex analytic space. Denoting by $p:X\to X/G$ the projection, the structure sheaf on $X/G$ is described as $\Oh_{X/G}(U):=\Oh_{X}(p^{-1}(U))^G$. The singularities of $X/G$ are contained in the collection of all points with nontrivial stabilizer. On the other hand, the singularities in normal analytic spaces have codimension at least two and this inclusion is in general not an equality as the following example illustrates:
\begin{ex}\label{ex: C/pm1}
Let $X=\C$, $G={\pm 1}$ acting by multiplication. Then the complex space $X/G$ is smooth: The map $z\mapsto z^2$, identifies it with $\C$.
\end{ex}

This example contains in some sense all the difference between the singular set and the set with nontrivial stabilizers: Recall that a matrix $M\in GL_n(\C)$ of finite order is called a quasi-reflection if $\dim\ker (M-\id)\geq n-1$, i.e, there is at most one eigenvalue not equal to 1 (including multiplicities), which then is a root of unity. A group $G\subseteq GL_n(\C)$ is called small if it does not contain any quasi-reflections. Note that quasi-reflections have fixed point set of codimension $1$. On the other hand, as a normal complex analytic space $X$ has singularities only in codimension $\geq 2$.
The following is a standard result, essentially the Chevalley-Shephard-Todd Theorem, see e.g. \cite[§7.2]{Ben93}, \cite{BG08}. 

\begin{prop}\label{prop: quasi-reflections}
    Let $X$ be a complex manifold with an action of a finite group $G$. The quotient $X/G$ is a manifold if and only if around every point there are invariant charts in which $G$ acts by quasi-reflections.
\end{prop}


\subsection{Smooth differential forms}
Given a quotient $p:X\to X/G$ of a complex manifold by a finite group, just as one defines holomorphic functions, one may also define sheaves of invariant holomorphic forms and smooth $\C$-valued functions and forms on the complex space $X/G$ by $\Omega_{X/G}^{inv}(U):=\Omega_{X/G}(p^{-1}(U))^G$ and $\cA_{X/G}^{inv}(U):=\cA_X(p^{-1}(U))^G$. As on usual differential forms, these carry a real structure, a bigrading by type, differentials $\del,\delbar$ and a wedge product such that on every open set $U\subseteq X/G$, the global sections $(\cA_{X/G}^{inv}(U),\del,\delbar,\wedge)$ form a bigraded, bidifferential algebra with real structure. Now, if the complex space $X/G$ is again a complex manifold, that gives us a priori a second notion of holomorphic, resp. smooth forms, namely those which are intrinsically defined from the manifold structure. Denote these by $\Omega_{X/G}$, resp. $\cA_{X/G}$. The (local) pullbacks induce an injective map of sheaves $\Omega_{X/G}\to \Omega_{X/G}^{inv}$ and $\cA_{X/G}\to \cA_{X/G}^{inv}$, compatible with real structure, bigrading, differentials and products. In general these maps are not isomorphisms:

\begin{ex}
Consider $X=\C$ and $G=\{\pm 1\}$ as in Example \ref{ex: C/pm1}. The function $(z\mapsto|z|^2)$ is a global section of $(\cA_{X/G}^{inv})^{0,0}$, but it is not a global section of $\cA_{X/G}^{0,0}$ since it is not the pullback of any smooth function on $\C$ via $(z\mapsto z^2)$.
\end{ex}

However, we have:

\begin{prop} Let $X/G$ be non-singular. Then $\Omega^{inv}_{X/G}\cong \Omega_{X/G}$ and both $((\cA_{X/G}^{inv})^{p,\cdot},\delbar)$ and $(\cA_{X/G}^{p,\cdot},\delbar)$ are fine resolutions of $\Omega^p_{X/G}$.
\end{prop}
\begin{proof}
For holomorphic forms, this is a consequence of Proposition \ref{prop: quasi-reflections} and for $(\cA_{X/G}^{p,\cdot},\delbar)$ this is the standard Poincar\'e Lemma. So, we only have to consider $(\Omega^{inv}_{X/G})^p\hookrightarrow((\cA_{X/G}^{inv})^{p,\cdot},\delbar)$. To show that the sheaves $(\cA_{X/G}^{inv})^{p,q}$ are fine, we need to find partitions of unity for $1\in (\cA_{X}^{l.i.})^{0,0}$ subordinate to any open cover of $X/G=\bigcup V_i$. This corresponds to finding $G$-equivariant partitions of unity subordinate to the cover for $X=\bigcup p^{-1}(V_i)$, where $p:\tilde{U}\to X$ is the projection map. These exist by an averaging argument. Next, we have to prove the exactness of 
\[
0\to (\Omega_{\tilde U/G}^{inv})^p\to (\cA_X^{inv})^{p,0}\to (\cA_X^{inv})^{p,1}\to\cdots
\]
This follows from the usual Poincar\'e Lemma on $X$ because taking $G$-invariants is exact (on stalks, this is a sequence of $\C$-vector spaces and one can average a given primitive to be $G$-invariant.)
\end{proof}
Let us denote by $A_X^G=(\cA_{X/G}^{inv}(X/G),\del,\delbar)$ and $A_{X/G}=(\cA_{X/G}(X/G),\del,\delbar)$ the double complexes of global sections. Then we obtain:
\begin{cor}\label{cor: invariant bicomplex smooth case}
Let $X/G$ be non-singular. The inclusion of double complexes $A_{X/G}\hookrightarrow A_{X}^{G}$ is a bigraded quasi-isomorphism.
\end{cor}

We note that all the discussion here was local, and so one may immediately generalize the above statements to orbifolds $\cY$ (spaces which are locally of the form $X/G$), and the analogously defined sheaves and spaces of 'locally invariant forms' $\cA^{l.i.}_{\cY}$:

\begin{prop} Let $\cY$ be an orbifold with non-singular underlying complex space $Y$. Both $((\cA_\cY,{l.i.})^{p,\cdot},\delbar)$ and $(\cA_{Y}^{p,\cdot},\delbar)$ are fine resolutions of $\Omega^p_{Y}$.
\end{prop}
\begin{cor}
On an orbifold $\cY$ with smooth underlying complex space $Y$, the inclusion of double complexes $A_X\hookrightarrow A_{\cY}^{l.i.}$ is a bigraded quasi-isomorphism.
\end{cor}
In this article, all orbifolds we encounter will be of `global quotient type', i.e. of the form $X/G$ as before, so we omit a more detailed discussion.
\section{Proof of Theorem \ref{thmintro: resolution}}

\begin{lem}
    Let $X$ be a compact complex manifold and $\sigma$ an automorphism of finite order with fixed point set $X^\sigma\subseteq X$. Then,
    \begin{enumerate}
        \item\label{it: submanifold} $X^\sigma\subseteq X$ is a complex submanifold,
        \item\label{it: lift} given an automorphism $\tau$ commuting with $\sigma$ and a $\tau$-stable submanifold $Z\subseteq X^\sigma$, $\tau$ lifts to an automorphism of $Bl_{Z}X$.
    \end{enumerate}
\end{lem}
\begin{proof}
   Around a point $x\in X^\sigma$, one may find an equivariant chart $(U,\varphi)$ such that on $\varphi(U)$, $\sigma$ acts by a diagonal matrix of which the first $k$ eigenvalues are not equal to $1$ and the next $n-k$ are equal to $1$. Thus, in this chart, the fixed point set is given by $z_1=\dots=z_k=0$. This proves \ref{it: submanifold}. Statement \ref{it: lift} follows from the universal property of the blow-up.\end{proof}

\begin{lem}\label{lem: 2,3}
In the situation of the previous Lemma, if the order of $\sigma$ is $2$ or $3$, there is a bimeromorphic $\sigma$-equivariant map $\hat X\to X$ which is finite sequence of blow-ups in smooth, $\sigma$-invariant centers such that $\hat X/\sigma$ is smooth. 
\end{lem}
\begin{proof}
    Let us work locally around a fixed point of $\sigma$ and assume we are in a chart s.t. $\sigma$ acts diagonally with $k\geq 2$ nontrivial eigenvalues. We show that by blowing up we can reduce the $k$ until all local actions are by quasi-reflections, so that the quotient is smooth.
    
    Assume first the order of $\sigma$ is $2$ and the action is by a diagonal matrix  $(-1,\dots,-1,1,\dots,1)$ on (an open set in) $\C^n$, with coordinates $z_1,\dots,z_n$. We may assume $k\geq 2$, otherwise $\sigma$ acts as a quasi-reflection. Then the blow up of the fixed point set is given by \[\{((z_1,\dots,z_l),[x_1,\dots,x_k])\in \C^n\times \CP^{k-1}\mid z_rx_s=z_sx_r\},\] which is covered by the $k$ invariant open sets $V_i:=\{x_i\neq 0\}$ with coordinates $v^i_j=x_j/x_i$ for $i\neq j$ and $v^i_i=z_i$. Thus, the induced action of $\sigma$ on $V_i$ is still diagonal, and in fact a reflection with only nontrivial eigenvalue in the $i$-th position. As such, the quotient is smooth.

    For the case that the order of $\sigma$ is $3$, we argue similarly. Write $\zeta=\exp(2\pi i/3)$. Lets assume we are in a linear diagonal chart around a fixed point and $\sigma$ has $k_1$ eigenvalues $\zeta$ and $k_2$ eigenvalues $\zeta^2$, so after possibly reordering the coordinates we may assume $\sigma$ is a diagonal matrix with entries $(\zeta,\dots,\zeta,\zeta^2,\dots,\zeta^2,1,\dots,1)$. Again, denoting $V_i$ as above, we see that for $i=1,\dots,k_1$, the induced action is given by a diagonal matrix $(1,\dots,\zeta,\dots1,\zeta,\dots,\zeta,1,\dots,1)$ with $k_2+1$ eigenvalues $\zeta$ and for $i=k_1+1,\dots,k$, the action is given by a diagonal matrix $(\zeta^2,\dots,\zeta^2,1,..,1,\zeta^2,1,\dots,1)$ with $k_1+1$ eigenvalues $\zeta$. So we see that if $k_1$ or $k_2$ is zero the induced action is by quasi-reflections. If not, we can blow up the new fixed points again and repeat the argument.
\end{proof}
\begin{rem}Lemma \ref{lem: 2,3} is a folklore result and it is well-known that it fails in general for cyclic groups of order $\geq 4$, see e.g. \cite[p.31f]{Oda88}.
\end{rem}

\begin{proof}[Proof of Theorem \ref{thmintro: resolution}]

Without loss of generality, we may assume $G$ acts effectively. Pick an element in $\sigma\in G$ of order $2$ or $3$. Then by the previous Lemmata, there exists a finite sequence of blow-ups in smooth invariant centers $\hat X\to X$ such that $\hat X/\sigma$ is smooth and a lift of the action of $G$ to $\hat X$ making this an equivariant map. Now, by the blow-up formula, \cite{Ste21b}, \cite{ASTT}, we have a $G$-equivariant bigraded quasi-isomorphism $A_{\hat X}\simeq A_X\oplus R$ with $A$ as in the theorem. Hence, taking $\sigma$-invariants, and using Cor. \ref{cor: invariant bicomplex smooth case} there is a $G/\sigma$-equivariant quasi-isomorphism $A_{\hat X/\sigma}\simeq A_X^{\sigma=\id}\oplus R'$ with $R'$ as in the Theorem. Furthermore, $\hat X/\sigma$ has an action of $G/\sigma$, which has lower order, so the the theorem follows by induction.    
\end{proof}

We will carry out the procedure in this proof explicitly in the first construction below.
\begin{rem}
In the proof, the resolusion map $\tilde X\to X/G$ is the composition of a sequence of maps $f_{i}:X_{i+1}\to X_{i}$ such that each $X_{i}$ is a global quotient of a nonsingular space by $G/H_{i}$ for a subgroup $H_{i}\subset G$ and each map $f_{i}:X_{i+1}\to X_{i}$ comes from a blow up at $G/H_{i}$-invariant center.
We should notice that $f_{i}:X_{i+1}\to X_{i}$ may not be extended to a $G$-equivariant blow-up.
We may observe this matter in Construction I.
\end{rem}

\begin{rem} The addendum in Theorem \ref{thmintro: resolution} concerning the fundamental group follows from a standard van-Kampen argument on blow-ups. We note such a statement is true in greater generality, \cite[Thm. 7.5.]{kol}, \cite{tak}.
\end{rem}

\begin{rem}
The structure of $R$ in \eqref{eqn: A=AG+R} in Theorem \ref{thmintro: resolution} (resp. by \eqref{eqn: (0,p)} and \eqref{eqn: (p,0)}) one sees that $X/G$ (resp. any space $Y$ with rational singularities) is further holomorphically simply connected in the sense of \cite{Ste23} (i.e. connected and $h^{1,0}=h^{0,1}=0$) if and only if the (resp. any) resolution is holomorphically simply connected. The manifolds in constructions \ref{constr: p1ddbar}, \ref{constr: ddc+3}, \ref{constr: FSSpi1} below are also holomorphically simply connected. Note that there are simply connected spaces which are not holomorphically simply connected (e.g. the Calabi-Eckmann structures on products of spheres of odd dimension $\geq 3$).
\end{rem}

\begin{rem}
     Quotient singularities are rational, and so for \textit{any} resolution $\pi:\tilde X\to X/G$, one has 
     \begin{equation}\label{eqn: (0,p)}
         H^p(X/G,\Oh_{X/G})\cong H^p(\tilde X,\Oh_{\tilde X}).
     \end{equation} 
     Further, if $j:(X/G)^{reg}\to X/G$ is the inclusion, one may also define the differential forms on $X/G$ as $\Omega_{X/G}:=j_*\Omega_{(X/G)^{reg}}$. Then, again for \textit{any} resolution, one has 
     \begin{equation}\label{eqn: (p,0)}
         H^0(X/G,\Omega_{X/G}^p)\cong H^0(\tilde X,\Omega_X^p).
     \end{equation}
     In both cases, the analogous isomorphisms in $H^{0,p}$ and $H^{p,0}$ continue to hold if one replaces $X/G$ by any complex space $Y$ with rational singularities \cite{KS21}, \cite{GKK10}.
     
    For the specific resolutions from Theorem \ref{thmintro: resolution}, the isomorphism \eqref{eqn: A=AG+R} recovers \eqref{eqn: (0,p)} and \eqref{eqn: (p,0)} but, because $R$ satisfies further conditions than just $R^{p,0}=R^{0,p}=0$, \eqref{eqn: A=AG+R} contains finer information, also `away from the boundary': For example, the numbers $b_3^{2,2}$ are preserved under the resolution, and so is the (non-)existence of nontrivial Fr\"olicher differentials landing in degrees $(n-1,1)$. This kind of finer information is crucially used in the following constructions. In view of \eqref{eqn: (p,0)} and \eqref{eqn: (0,p)} holding for all rational singularities and resolutions, it seems to be an interesting question in what generality the conclusion \eqref{eqn: A=AG+R} of Theorem \ref{thmintro: resolution} remain valid.
\end{rem}

\section{Applications of Theorem \ref{thmintro: resolution}}
\subsection{Construction \ref{constr: p1ddbar}}\label{sec: constr p1ddbar}
We will construct the manifolds $M_b$ as the resolution of Theorem \ref{thmintro: resolution} of a singular quotient of a certain solvmanifold $X_b$.

Let $G:=\C\ltimes _{\phi}\C^{2}$ such that \[
\phi(x+\sqrt{-1}y)=\left(
\begin{array}{cc}
e^{x}& 0  \\
0&    e^{-x}  
\end{array}
\right).\]
Quotients of $G$ by lattices $\Gamma=(a\Z+bi\Z)\ltimes_\phi \Delta$ are known as completely solvable Nakamura-manifolds and their cohomologies can be computed in terms of a finite dimensional sub double complex $C_b\subseteq A_{\Gamma\backslash G}$, see \cite{AK}. In general, the result will depend on the lattice $\Gamma$, more precisely, on the value of $b\in \R-\{0\}$. We will now construct a particular family of lattices $\Gamma_b$, such that the corresponding manifolds $X_b:=\Gamma_b\backslash G$ admit an automorphism of order $4$.

Let \[
A:=\left(
\begin{array}{cc}
2& 1  \\
1&    1  
\end{array}
\right).\] 
We diagonalize
\[
P^{-1}AP=\left(
\begin{array}{cc}
 \frac{3+\sqrt{5}}{2}& 0  \\
0&    \frac{3-\sqrt{5}}{2} 
\end{array}
\right)\quad \text{with} \quad P:=\left(
\begin{array}{cc}
1& \frac{-1+\sqrt{5}}{2}  \\
\frac{-1+\sqrt{5}}{2}&    -1  
\end{array}
\right).\]
Take $a\in \R$ so that $e^{a}=\frac{3+\sqrt{5}}{2}$.
Consider the discrete  subgroup \[
\Delta:=\left\{P^{-1}\left(
\begin{array}{cc}
m  \\
n 
\end{array}\right): m,n \in \Z[i]\right\}.
\]
Then for any  $k\in \Z$ and $y\in \R$, $\phi(ka+\sqrt{-1}y)P^{-1}=P^{-1}A^{k}$
and hence $\phi(k+\sqrt{-1}y)(\Delta)\subset \Delta$.
For  $\Lambda =a\Z+b\sqrt{-1}\Z$, with $\R\ni b\not=0$, $\Gamma=\Lambda\ltimes_{\phi}\Delta\subset G$ is a lattice in $G$. We denote by $X_b=\Gamma\backslash G$ the resulting solvmanifold. We consider \[
\sigma:= \left(
\begin{array}{ccc}
-1& 0&0  \\
0&    0&1\\
0&-1&0 
\end{array}
\right)
\] as a biholomorphic automorphism of $G$.
Since 
\[\left(
\begin{array}{cc}
    0&1\\
-1&0 
\end{array}
\right)P=-P\left(
\begin{array}{cc}
    0&1\\
-1&0 
\end{array}
\right),\]
we have $\sigma (\Gamma)=\Gamma$. The cyclic group $\langle \sigma\rangle \cong \Z/4\Z$ acts on the solvmanifold $X_{b}$ biholomorphically and we obtain a (singular) complex quotient space $X_{b}/\langle \sigma\rangle$ depending on $b\in \R-\{0\}$.
For  $\langle \sigma^{2}\rangle \cong \Z/2\Z$, we also obtain a complex quotient space $X_{b}/\langle \sigma^2\rangle$ depending on $b\in \R-\{0\}$.\\

There is a finite dimensional double complex $C_b\subseteq A_{X_b}$ such that the inclusion is a bigraded quasi-isomorphism. It was explicitly given in \cite{AK}. For us, only the $\sigma$-invariant part will be relevant, which we list below. Following \cite{AK}, we distinguish three cases, depending on the value of $b$:
\begin{enumerate}
\item\label{item:nakamura-1} $b=2m\pi$ for some integer $m\in\Z$;
\item\label{item:nakamura-2}  $b=(2m+1)\pi$ for some integer $m\in\Z$; 
\item\label{item:nakamura-3}$b\not=m\pi$ for any integer $m\in\Z$.
\end{enumerate}


\begin{prop}
For any $b\in\R-\{0\}$, the complex $C_b^{\sigma=\id}$ satisfies the page-$1$-$\del\delbar$-property. It satisfies the usual $\del\delbar$-property only in case \ref{item:nakamura-3}. In cases \ref{item:nakamura-1} and \ref{item:nakamura-2}, there are nonzero differentials $E_1^{1,1}\to E_1^{2,1}$ and $E_1^{1,2}\to E_1^{2,2}$.
\end{prop}
\begin{proof}
We list the forms belonging to $C_b^{\sigma=\id}$ in Table \ref{table:c-nakamura-1} in the appendix. In case \ref{item:nakamura-3}, the differential on $C_b^{\sigma=\id}$ is zero, hence it trivially satisfies the $\del\delbar$-property. In cases \ref{item:nakamura-1} and \ref{item:nakamura-2} on the other hand, we have e.g. $d_1[\esp^{-2z_{1}}\de z_{2\bar2}+ \esp^{2z_{1}}\de z_{3\bar3}]\neq 0\in E_1^{2,1}$. In Figure \ref{fig: Cbsigma} we give a schematic picture of the indecomposable summands of $C_b^{\sigma=\id}$.
\end{proof}

\begin{figure}
    \centering
    \includegraphics{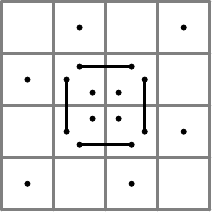}\qquad\qquad\includegraphics{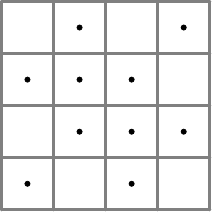}
    \caption{The indecomposable summands in $C_b^{\sigma=\id}$ for $b\in\pi\Z$, resp. $b\not\in\pi\Z$.}
    \label{fig: Cbsigma}
\end{figure}

\begin{prop}
 The complex space $X_{b}/\langle \sigma\rangle$ is simply connected.
\end{prop}

\begin{proof}
We consider the fiber bundle $\Delta \backslash \C^{2}\to X_{b}\to \Lambda  \backslash \C$.
By the splitting $G=\C\ltimes _{\phi}\C^{2}$, we have the section  $\Lambda  \backslash \C\to X_{b}$.
The $\sigma^2$-action is compatible with this bundle structure.
 The action on the fiber is the standard involution $-1$ on $\Delta \backslash \C^{2}$ and  the one on the base is trivial.
 Thus $X_{b}/\langle\sigma^2 \rangle$ is a fiber bundle over a torus $\Lambda  \backslash \C$ with the simply connected fiber $\Delta \backslash \C^{2}/(-1)$.
 We have $\pi_{1}(X_{b}/\langle \sigma^2\rangle)\cong \pi_{1}(\Lambda  \backslash \C)=\Lambda$.
Consider 
$X_{b}/\langle \sigma\rangle=(X_{b}/\langle \sigma^2\rangle)/\langle [\sigma]\rangle$.
The action of $\langle [\sigma]\rangle$ is compatible  with the  bundle structure on $X_{b}/\langle \sigma^2\rangle$.
The action on the base is  the standard involution $-1$ on $\Lambda  \backslash \C$.
Since $\Lambda  \backslash \C/ (-1)$ is simply connected,
the quotient $X_{b}/\langle \sigma^2\rangle\to X_{b}/\langle \sigma\rangle$ sends $\pi_{1}(X_{b}/\langle \sigma^2\rangle)\cong\pi_{1}(\Lambda  \backslash \C)$ to the trivial group.
Thus, by \cite[Corollary 6.3]{Bre},  $\pi_{1}(X_{b}/\langle \sigma\rangle)$ is trivial.
\end{proof}

We now take $M_b$ to be resolutions as in Theorem \ref{thmintro: resolution} of the spaces $X_b/\langle\sigma\rangle$. Since this modifies the bigraded quasi-isomorphism type of $A_{X_b/\sigma}\simeq C^{\sigma=\id}_b$ only by bicomplexes of points and curves, all of which satisfy the $\partial\bar\partial$-Lemma, we conclude. For completeness, we give a more explicit description of the geometry of $X_b$ and the resolution procedure below:

\subsubsection{Geometry of $X_b$}
\begin{lem}\label{lem: double mapping torus}($X_b$ as a double mapping torus)
The map $\id\times (P\cdot\_)$ induces an identification 
\[
X_b\cong \C\times_{\Z^2}(\Z[i]^2\backslash\C^2),
\]
where the basis element $e_1=(1,0)$ of $\Z^2$ acts by $+a$ on $\C$ and by $A\cdot\_$ on the right and the basis element $e_2=(0,1)$ acts by $+ib$ on $\C$ and by the identity on the right. This identification is compatible with the fibre bundle projections to $\Lambda\backslash\C$ on both sides. The action of $\sigma$ on the left is identified with the action of the matrix $\sigma^3=\sigma^{-1}$ on the right. 
\end{lem}
\begin{proof}
Follows from the definitions.
\end{proof}

\begin{lem}[description of the fixed point set]\label{lem: fixed points}\,
\begin{enumerate}
    \item The action of $\sigma^2$ on $X_b$ fixes exactly six curves $C_0,...,C_5\subseteq X_b$. Under the projection $p:X_b\to \Lambda\backslash\C$, the curve $C_0$ maps isomorphically to the base torus and $C_1$,...,$C_5$ are $3$--sheeted connected coverings. In particular, the fixed point set intersects the fibre of $p$ over each point in sixteen points. In appropriate local coordinates, the action around each fixed point looks like that of the diagonal matrix $\Delta(1,-1,-1)$ on $\C^3$.
    \item The action of $\sigma$ interchanges $C_1$ and $C_2$ and $C_4$ and $C_5$. It restricts to an action on $C_0$ and $C_1$. The fixed point set consists of sixteen points, given by the intersection of $C_0\cup C_3$ with the fibres over $P_0:=[0]$, $P_1:=[a/2]$, $P_2:=[ib/2]$ and $P_3:=[a/2+ib/2]$. In appropriate local coordinates, the action around each fixed point looks like that of the diagonal matrix $\Delta(-1,i,-i)$ on $\C^3$.
\end{enumerate} 
\end{lem}
\begin{proof}
We will identify $X_b\to\Lambda\backslash\Lambda$ with a double mapping torus as in Lemma \ref{lem: double mapping torus}. The action of $\sigma^2=\Delta(1,-1,-1)$ on the base is trivial, so to compute the fixed points, we first consider its action in each fibre. There it is the multiplication by $(-1)$ on a standard torus $(\Z[i])^2\backslash\C^2$ and the $16$ fixed points are given by the set
\[
F:=\left\{P=(x,y)\mid x,y\in \left\{0,\frac 1 2,\frac i 2,\frac{1+i}{2}\right\}\right\}.
\]
If the monodromy $A$ were trivial, we would thus have $16$ fixed curves of $\sigma^2$, each mapping isomorphically to the base. However, it is not trivial and so we have to compute its action on $F$. Doing so, we find for example that it acts trivially on $(0,0)$ but has order three on other points, e.g. 
\[\left(\frac 1 2,0\right)\mapsto \left(1,\frac 12\right)\sim \left(0,\frac 12\right)\mapsto \left(\frac 12,\frac 12\right)\mapsto \left(\frac 3 2, 1\right)\sim\left(\frac 12,0\right)\]
etc. In summary, we obtain, slightly abusing notation:
\begin{align*}
C_0&:=\Lambda\backslash\C\times \{(0,0)\}\\
C_1&:=\Lambda\backslash\C\times \left\{\left(\frac 12,0\right),\left(0,\frac 12\right),\left(\frac 12,\frac 12\right)\right\}\\
C_2&:=\Lambda\backslash\C\times \left\{\left(\frac i2,0\right),\left(0,\frac i2\right),\left(\frac i2,\frac i2\right)\right\}\\
C_3&:=\Lambda\backslash\C\times\left\{\left(\frac {1+i} 2,0\right), \left(0,\frac{1+i} 2\right), \left(\frac {1+i} 2,\frac{1+i} 2\right)\right\}\\
C_4&:=\Lambda\backslash\C\times\left\{\left(\frac 12,\frac i2\right),\left(\frac i2,\frac{i+1}2\right), \left(\frac {i+1}2, \frac 12\right)\right\}\\
C_5&:=\Lambda\backslash\C\times\left\{\left(\frac i2,\frac 12\right), \left(\frac 12,\frac{i+1}2\right),\left(\frac{i+2}2,\frac i2\right)\right\}
\end{align*}
We note that the projection to $\Lambda\backslash B$ is equivariant with respect to $\sigma$ and multiplication by $(-1)$. The latter has fixed points $P_0,...,P_3$. All fixed points of $\sigma$ on $X_b$ therefore have to lie in the fibres over the $P_i$. The remaining assertions are straightforward to check.
\end{proof}

\subsubsection{Explicit description of the resolution} 
  We will do this in two steps, first resolving $X_b/\langle \sigma^2\rangle$. For this, we consider the blow-up in all $C_i$ as in Lemma \ref{lem: fixed points}:
\[
X'_b:=Bl_{C_0\cup...\cup C_5}X_b\longrightarrow X_b
\]
This has an induced action of $\sigma^2$ and because it is in every fibre the Kummer-construction of a $K3$-surface, we obtain that $X'_b/\langle \sigma^2\rangle\to \Lambda\backslash\C$ is a locally trivial holomorphic $K^3$-bundle (in particular, it is a manifold).

Let us look more closely at the local structure of $X'_b/\langle\sigma\rangle$: around any point $P\in C_i$ we can find a $\sigma^2$-invariant open neighborhood around $P$ and a chart $\varphi$ to some neighborhood $U$ of $0$ in $C^3$ with coordinates $z_1,z_2,z_3$ such that $\varphi(p)=0$, $C_i\cong\{z_2=z_3=0\}\cap U$ and the action of $\sigma^2$ corresponds to the diagonal action $\Delta(1,-1,-1)$. Hence, a local model for $X'_b$ is given by:
\[
\{((z_1,z_2,z_3),[w_2:w_3])\in U\times \mathbb{P}^2\mid z_2w_3=z_3w_2\}
\]
with $\sigma$-action
\[
((z_1,z_2,z_3),[w_2:w_3])\longmapsto((z_1,-z_2,-z_3),[-w_2:-w_3]).
\]
Thus, in the standard chart $U_2:=\{w_2\neq 0\}$, with coordinates $(t_1, t_2, t_3)$ for $t_1=z_1,t_2=\frac{w_3}{w_2}, t_3=z_3$, $\sigma^2$ acts as
\[
(t_1,t_2,t_3)\mapsto (t_1,t_2,-t_2)
\]
So, we see explicitly that the quotient by $\sigma^2$ is smooth with local coordinates on $U_2/\sigma^2$ given by $(t_1,t_2,t_3^2)$. A similar discussion applies to the other chart $U_3=\{w_3\neq 0\}$.

With this at hand, we now turn to the induced action of $\sigma$ on $X'_b/\sigma^2$. Denote by $E_i$ the exceptional divisors corresponding to $C_i$. Since $\sigma$ interchanges $C_1$ and $C_2$ (resp. $C_4$ and $C_5$) the same holds for $E_1$ and $E_2$ (resp. $E_4$ and $E_5$). Thus, the fixed points can only lie on $E_1$ and $E_4$. Let us thus assume that in the construction of the previous paragraph $P\in C_1\cup C_3$ and $\varphi$ is defined on a $\sigma$ invariant open and identified the action of $\sigma$ with the diagonal action $\Delta(-1,i,-i)$. Then, a local model for the $\sigma$-action on $X'_b$ is:
\[
((z_1,z_2,z_3),[w_2:w_3])\longmapsto ((-z_1,iz_2,-iz_3,[iw_2:-iw_3])
\]
Thus the induced action on $X'_b/\sigma^2$ in the coordinates $z_1,t_2,t_3^2$ looks as follows:
\[
(z_1,t_2,t_3^2)\longmapsto (-z_1,-t_2,-t_3^2)
\]
(and similarly in the coordinates corresponding to $U_3$). Over every point $P_1,...,P_4\in \Lambda\backslash\C$, there thus lie four fixed points of $\sigma$, two on each of $E_1$ and $E_3$, with local action given by multiplication by $(-1)$.

If one blows up such a point with local action $(-1)$, a local model for the blow-up is given by
\[
\{((x_1,x_2,x_3),[y_1:y_2:y_3]\mid x_1y_2=x_2 y_1, x_1y_3=x_3y_1, x_2y_3=x_3y_2\}
\]
and the induced action in the coordinates on $U_1=\{y_1\neq 0\}$ given by $s_1=x_1,s_2=\frac{y_2}{y_1},s_3=\frac{y_3}{y_1}$ is given by
\[
(s_1,s_2,s_3)\longmapsto (-s_1,s_2,s_3).
\]
Hence, again the quotient is smooth (with coordinates $(s_1^2,s_2,s_3)$). More globally, if we denote by $Q_1,...,Q_{16}$ the fixed points of $\sigma$ on $X'_b/\sigma^2$ and set $\hat{X}_b:=Bl_{Q_1\cup...\cup Q_{16}}X'_b/\sigma^2\to X'_b/\sigma^2$, 
then the induced map
\[
\hat{X}_b/\sigma\longrightarrow X'_b/\sigma
\]
is a resolution of singularities. Let us summarize the situation in a diagram:
\[
\begin{tikzcd}
                        &X'_b\ar[r]\ar[d]    &X_b\ar[d]\\
    \hat{X}_b\ar[r]\ar[d]     &X'_b/\sigma^2\ar[r]\ar[d]          &X_b/\sigma^2\ar[d]\\
    \hat{X}_b/\sigma\ar[r]&X'_b/\sigma\ar[r]& X_b/\sigma~.
\end{tikzcd}
\]
Each leftmost horizontal map (resp. each composition of horizontal maps) is a resolution of singularities. Let us now use this to compute the bigraded quasi-isomorphism type of the double complex of forms on the manifold $\hat{X}_b/\sigma$:
\begin{align*}
    A_{\hat{X}_b/\sigma}&\simeq A_{\hat{X}_b}^{\sigma=\id}\\
                        &\simeq \left( A_{X'_b/\sigma^2}\oplus \bigoplus_{i=1}^{16}(A_{Q_i}[1,1]\oplus A_{Q_i}[2,2])\right)^{\sigma=\id}\\
                        &\simeq \left(A_{X'_b}^{\sigma^2=\id}\oplus \bigoplus_{i=1}^{16}(A_{Q_i}[1,1]\oplus A_{Q_i}[2,2])\right)^{\sigma=\id}\\
                        &\simeq \left( \left(A_{X_b}\oplus\bigoplus_{i=0}^{5}A_{C_i}[1,1]\right)^{\sigma^2=\id}\oplus \bigoplus_{i=1}^{16}(A_{Q_i}[1,1]\oplus A_{Q_i}[2,2])\right)^{\sigma=\id}\\
                        &\simeq C_b^{\sigma=\id}\oplus \left(\bigoplus_{i=0}^5 A_{C_i}[1,1]\right)^{\sigma=\id}\oplus\bigoplus_{i=1}^{16}\left(A_{Q_i}[1,1]\oplus A_{Q_i}[2,2]\right)
\end{align*}

Since points and curves satisfy the $\del\delbar$-property, we see that the manifolds $M_b:=\hat{X}_b/\sigma$ satisfy the statement in Construction \ref{constr: p1ddbar}.

\begin{rem}
Recall that a reduced  compact complex space is  in Fujiki's class $\mathcal C$ if it  admits  a compact K\"ahler modification.
By Hironaka's Flatting theorem, this condition is equivalent to being  a meromorphic image of  a compact K\"ahler space (see \cite[Theorem 5]{Var}).
The latter condition is the original definition  of Fujiki in \cite{Fuj}.

We notice that for any $b$,  $M_b:=\hat{X}_b/\sigma$ is not in Fujiki's class $\mathcal C$.
If $M_b:=\hat{X}_b/\sigma$ is in Fujiki's class $\mathcal C$, then   $X_{b}/\langle \sigma\rangle $ is a compact complex space in Fujiki's class $\mathcal C$  by the resolution  $M_b\to X_{b}/\langle \sigma\rangle $ and hence $X_{b}$ is a complex space  in  Fujiki's class $\mathcal C$ by \cite[Lemma 4.6]{Fuj} since the quotient map $X_{b}\to X_{b}/\langle \sigma\rangle $ is a K\"ahler morphism. This contradicts \cite[Theorem 3.3]{AK14} and \cite[Main Theorem]{Has}.

In \cite{Fri}, Friedman constructs compact complex three dimensional  $\partial\bar\partial$-manifolds with second Betti numbers$=0$ and   trivial canonical bundles by using generalized Clemens construction.
On our example $M_b$, for any $b$, we have $H^2_{dR}(M_b)\not=0$ and  $H^0(M_b, \Omega^3)=0$ (see Appendix B) in particular the canonical bundle of $M_b$ is non-trivial.

\end{rem}

\subsection{Construction \ref{constr: ddc+3}}
We proceed in a similar fashion to the previous section, constructing a simply connected, singular quotient of a complex $4$-fold by a finite group with invariant forms having the $dd^c+3$ property, then taking a resolution as in Theorem \ref{thmintro: resolution}. Since all points, curves and surfaces satisfy the $dd^c+3$-condition (c.f. \cite{SW22}), this is sufficient.

Consider the nilpotent Lie group 
\[
H=\left\{M\in Gl_4(\R)\mid M=\begin{pmatrix}
1 & x_1 & y_1 & x_3\\
0 & 1 & 0 & x_2\\
0 & 0 & 1 & y_2\\
0 & 0 & 0 & 1
\end{pmatrix}\right\}
\]
and set $G:=H\times \R^3$, where $\R^3$ has coordinates $(y_3,x_4,y_4)$. Let $H_\Z:=H\cap GL_4(\Z)$ and $\Gamma=H_\Z\times\Z^3\subseteq G$ a (automatically cocompact) lattice in $G$. Define $X:=\Gamma\backslash G$.
One obtains a basis of left-invariant differential forms $X^1,...,X^4,Y^1,...,Y^4$ on $G$ by

\[
Y^i:=dy_i,\qquad X^i:=\begin{cases}dx_i&i\neq 3\\-2dx_3+2x_1dx_2+2y_1dy_2&i=3\end{cases},
\]
so that $dX^3=2(X^1X^2+Y^1Y^2)$ and $dX^i=dY^i=0$ otherwise. We may define a left-invariant complex structure $J$ on $G$ by $J(X^i)=Y^i$. Putting $\omega^i:= X^i+iY^i$, we thus obtain a complex of $\C$-valued left-invariant forms given by
\[
A:=\Lambda(\omega^1,...,\omega^4,\bar\omega^1,...,\bar\omega^4)\quad\text{s.t.} \quad d\omega^3=d\bar\omega^3=\omega^{1\bar 2}+\omega^{\bar 1 2}, ~d\omega^k=d\bar\omega^k=0\text{ else.}
\]
The complex structure thus defined is nilpotent, and therefore the inclusion $A\subseteq A_X$ is a bigraded quasi-isomorphism \cite{CFGU}.

Define an automorphism of $G$ by 
\[
\sigma(x_1,y_1,x_2,y_2,x_3,y_3,x_4,y_4)=(-y_1,x_1,y_2,-x_2,-x_3,-y_3,-x_4,-y_4).  
\]
On $A$, this induces the action
\[
\sigma(\omega^1)=i\omega_1,~\sigma(\omega^2)=-i\omega^2,~\sigma(\omega^3)=-\omega^3,~\sigma(\omega^4)=-\omega^4.
\]

\begin{prop}
    The cdga $A^{\sigma=\id}$ satisfies the $dd^c+3$ condition, but not the $dd^c$-condition.
\end{prop}

\begin{proof}
    Since everything is explicitly defined, it is a routine calculation to compute the decomposition of $A$ into indecomposables. We give an explicit vector space basis for $A^{\sigma=\id}$ in Table \ref{table: ddc+3} in the Appendix and the full decomposition in Figure \ref{fig: Afix}. One can also be slightly more efficient as follows: First check that $A':=\Lambda(\omega^1,\omega^2,\omega^3,\omega^{\bar1},\omega^{\bar2},\omega^{\bar3})$ satisfies the $dd^c+3$-condition. Then necessarily $A'\otimes\Lambda(\omega^4,\omega^{\bar 4})$ satisfies the $dd^c+3$-condition since it is a tensor product with a $dd^c$-algebra. By Lemma \ref{lem: invariants}, $A^{\sigma=\id}$ satisfies the $dd^c+3$-condition. It does not satisfy the usual $dd^c$-condition since e.g. the closed form $\omega^{34}+\omega^{\bar 3 4}$ is not cohomologous to a sum closed forms of pure-degree.
\end{proof}

\begin{figure}
    \centering
    \includegraphics{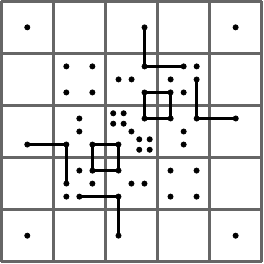}
    \caption{The indecomposable summands of $A^{\sigma=\id}$.}
    \label{fig: Afix}
\end{figure}



\begin{prop}
 The complex space $X/\langle\sigma \rangle$ is simply connected. 
\end{prop}
\begin{proof}
Take $(0,0,0,0,0,0,0,0)$ as the base point.
We can regard $\pi_{1}X=\Gamma$.
We have the generators 
$\gamma_{1}=(1,0,0,0,0,0,0,0),\dots ,\gamma_{8}=(0,0,0,0,0,0,0,1)$.
By \cite[Corollary 6.3]{Bre}, for the quotient $q:X\to X /\langle\sigma \rangle$, it is sufficient to show $q_{*}(\gamma_{i})$ is the unit element in $\pi_{1}(X/\langle\sigma \rangle)$.
Since $\gamma_{5},\dots, \gamma_{8}$ form a lattice in $\R^{4}=\{(0,0,0,0,x_3,y_3,x_4,y_4)\}$,
by the standard argument $q_{*}(\gamma_{i})$ is trivial for $i=5,6,7,8$.

We consider the quotient map $q^{'}:X\to X /\langle\sigma^2 \rangle$.
It is sufficient to prove $q^{'}_{*}(\gamma_{i})$ is the unit element in $\pi_{1}(X/\langle\sigma^2 \rangle)$ for $i=1,2,3,4$.
Consider the lattice $\Gamma^{'}$ generated by $\gamma_{1}, \gamma_{2}$ in $\R^{2}=\{(x_{1},y_{1},0,0,0,0,0,0)\}$.
Then the fiber bundle $G/\Gamma\to \T^{'}= \R^{2}/\Gamma^{'}$ has  the section $ \T^{'}\to G/\Gamma$.
$\sigma^2$ commutes with the bundle structure.
By the standard argument, $\pi_{1}(\T^{'}/\langle\sigma^2 \rangle)$ is trivial and we can say that $q^{'}_{*}(\gamma_{i})$ is the unit element for $i=1,2$. Taking care of $\gamma_{3}, \gamma_{4}$ with a similar argument, the proposition follows. \end{proof}

One may now conclude again by applying Theorem \ref{thmintro: resolution} to $X/\langle \sigma\rangle$.

\subsection{Construction \ref{constr: FSSpi1}}\label{sec: FSS}
\subsubsection{Review of Bigalke-Rollenske's examples}

In \cite{BR14}, L. Bigalke and S. Rollenske produce, for any $n$, a complex nilmanifold $X_n=G_n/\Gamma_n$ such that there is a nonzero differential on page $E_n$ of the Frölicher spectral sequence. A basis for the left-invariant $(1,0)$-forms in these examples is given by:
\[
dx_1,...,dx_{n-1},dy_1,...,dy_{n},dz_1,...,dz_{n-1}, \omega_1,...,\omega_n
\]
and the differential is zero except on the $\omega_i$'s, where it is given by (for brevity, we omit all wedge-products)
\begin{align*}
d\omega_1&=-d\bar{y}_1dz_1\\
d\omega_i&=dx_{i-1}dy_{i}+dy_{i-1}d\bar{z}_{i-1}~.
\end{align*}
Let \begin{align*}
\beta_1&:=\bar\omega_1d\bar{z}_2...d\bar{z}_{n-1}\\
\beta_2&:=\omega_2d\bar{z}_2...d\bar{z}_{n-1}\\
\beta_k&:= dx_1...dx_{k-2}\omega_k d\bar{z}_k...d\bar{z}_{n-1}\\
\beta_n&:=dx_1...dx_{n-2}\omega_n~.
\end{align*}
The differential on page $E_n$ then corresponds to the subcomplex $B$ given as follows:
\[
\begin{tikzcd}
	0&&&&\\
	\langle\beta_1\rangle\ar[r]\ar[u]&\langle\del\beta_1\rangle&&&\\
	&\langle\beta_2\rangle\ar[r]\ar[u]&\langle\del\beta_2\rangle&&\\
	&&..&{}&\\
	&&&\langle{\beta_n}\rangle\ar[u]\ar[r]&\langle\del\beta_n\rangle.
\end{tikzcd}
\]
This is indeed a subcomplex as
\begin{align*}
\delbar\beta_1&=0\\
\del\beta_1=-\delbar{\beta}_2&=-dyd\bar{z}_1...d\bar{z}_{n-1}\\
\del\beta_k=-\delbar\beta_{k+1}&=(-1)^{k-2}dx_1...dx_{k-1}dy_kd\bar{z}...d\bar{z}_{n-1}\\
\del\beta_n&=(-1)^{n-2}dx_1...dx_{n-1}dy_n~.
\end{align*}

The proof in \cite{BR14} shows that this subcomplex is even a direct summand in the double complex of all forms. Therefore the Frölicher spectral sequence of $X_n$ contains the Frölicher spectral sequence of $B$ as a direct summand and in $B$ one has $d_n([\beta_1])=[\beta_n]\neq 0$, hence also in $E_n(X)$.

\subsubsection{Definition of a group action}
Define an automorphism of the algebra of left-invariant forms on $G_n$ via $\sigma(dx_k)=idx_k,\sigma(dy_k)=idy_k, \sigma(dz_k)=idz_k$ and $\sigma(\omega_k)=-\omega_k$. This corresponds to an automorphism of the Lie algebra of $G_n$ (since it is compatible with $d$) and hence induces one of $G_n$. Because it respects the rational structure $G_n$ given by the real and imaginary parts of the $dx_k,dy_k,dz_k,\omega_k$, it respects a lattice of $G_n$. We have $\sigma(\beta_k)=-i^{n-2}\beta_k$, so the subcomplex $B$ is contained in the complex of invariant forms whenever $n\equiv 0(4)$. 

\begin{prop}
 The complex space $\sigma\backslash G_n/\Gamma_n$ is simply connected.
\end{prop}
\begin{proof}
Consider the coordinate
\[
x_1,...,x_{n-1},y_1,...,y_{n},z_1,...,z_{n-1}, w_1,...,w_n
\]
as \cite{BR14}.
Take $(0\dots,0)$ as the base point.
We can regard $\pi_{1} (G_n/\Gamma_n)=\Gamma_{n}$. By \cite[Corollary 6.3]{Bre}, for the quotient $q:X\to X /\langle\sigma \rangle$, it is sufficient to show $q_{*}(\gamma)$ is the unit element in $\pi_{1}(X/\langle\sigma \rangle)$ for any $\gamma\in \Gamma_n$. We have the fiber bundle  $ G_n/\Gamma_n\to \C^{3n-2}/\Gamma^{\prime}$ with a fiber $\C^{n}/\Gamma^{\prime\prime}$. Arguing on the fiber at $0\in \C^{3n-2}/\Gamma^{\prime}$, we can easily check that $q_{*}(\gamma)$ is the unit element for any $\gamma\in \Gamma^{\prime\prime}$. Consider the subgroup $\Gamma_{x_{i}}\subset \Gamma_{n}$ corresponding to $x_{i}\in \Z$ or $x_{i}\in i\Z$, other parameters$=0$.
Then we have the fiber bundle $ G_n/\Gamma_n\to \R^{1}/\Gamma_{x_{i}}$ with the $0$-section $\R^{1}/\Gamma_{x_{i}}\to G_n/\Gamma_n$.
Since $\sigma^2$ preserve this section and acts on $\R^{1}/\Gamma_{x_{i}}$ as antipodes.
Thus, we can say that 
$q_{*}(\gamma)$ is the unit element for any $\gamma\in \Gamma_{x_{i}}$.
By the same arguments on $y_{i}$, $z_{i}$, we can say that $q_{*}(\gamma)$ is the unit element for any $\gamma\in \Gamma$.\end{proof}

One may now apply again Theorem \ref{thmintro: resolution} to the quotient $\sigma\backslash G_n/\Gamma_n$ (for $n\equiv 0(4)$ and large enough) to conclude. 


\appendix

\section{Counterexamples to Popovici's SKT conjecture}\label{app: SKTvsFSS}
We keep the notations from section \ref{sec: FSS}. We were not able to write down SKT-metrics on Rollenske's examples directly, but let us modify them as follows: Keep all the forms as above and add $\theta_1,...,\theta_{n-1}$ and $\eta_1,...,\eta_{n-1}$ such that
\begin{align*}
d\theta_i=\delbar\theta_i&:=dy_id\bar{y}_i+dz_id\bar{z}_i\\
d\eta_i=\delbar\eta_i&:=dy_{i+1}d\bar{y}_{i+1}+dx_{i}d\bar{x}_{i}
\end{align*}
Since $d$ is zero on all factors on the right hand side, adding these, we still have $d^2=0$. In fact, the corresponding Lie-algebra with left-invariant complex structure is still two-step nilpotent with rational structure constants. In particular the corresponding Lie group $\tilde{G}_n$ admits some lattice $\tilde{\Gamma}_n$ by Malcev's theorem and the Dolbeault cohomology of $\tilde{X}_n:=\tilde{G}_n/\tilde{\Gamma}_n$ can be computed by left-invariant forms.\\

The claim is now that the subcomplex $B$ spanned by the $\beta_i$ and their differentials is still a direct summand (so that one still has a nonzero differential on $E_n$) and that there exists a left-invariant SKT-metric on $\tilde{X}_n$

Define a left-invariant hermitian metric $h$ by its associated form as:

\[
h:=\theta_1\bar\theta_1+\sum_{i=1}^n\omega_i\bar{\omega}_i+dy_id\bar{y}_i + \sum_{i=1}^{n-1}\frac{1}{2}(\eta_i\bar\eta_i+\theta_{i+1}\bar\theta_{i+1}) + dx_id\bar{x}_i+dz_id\bar{z}_i
\]

Using
\begin{align*}
\del\delbar(\omega_1\bar\omega_1)&=dy_1d\bar{y}_1dz_1d\bar{z}_1\\
\del\delbar(\omega_i\bar\omega_i)&=dy_{i-1}d\bar{y}_{i-1}dz_{i-1}d\bar{z}_{i-1}+dx_{i-1}d\bar{x}_{i-1}dy_id\bar{y}_i&(i\geq 2)\\
\del\delbar(\theta_i\bar\theta_i)&=-2dy_id\bar{y}_idz_id\bar{z}_i\\
\del\delbar(\eta_i\bar\eta_i)&=-2dx_{i}d\bar{x}_idy_{i+1}d\bar{y}_{i+1}
\end{align*}
one sees that this metric is SKT.\\

We still have to see that $d_n([\beta_0])\neq 0$ in the Fr\"olicher spectral sequence for $\tilde{X}_n$ or equivalently that $B$ is a direct summand in the larger complex $A$ of all left-invariant forms  on $\tilde{G}_n$. Define a subcomplex $\tilde{C}\subseteq A$ as follows: In every $(p,q)$ with 
\[(p,q)\notin S:=\{(0,n),(1,n),(1,n-1),...,(n,0),(n+1,0)\}
\]
(the positions of the $\beta_i$ and $\del\beta_i$), $C^{p,q}:=A^{p,q}$. In every bidegree $(p,q)\in S$, we set $C^{p,q}\subseteq A^{p,q}$ to be the subspace of all left-invariant forms generated by all elementary wedges of the basis elements $dx_i,dy_i,dz_i,\omega_i,\eta_i, \theta_i$ and their conjugates, except for the $\beta_i$ or $\del\beta_i$ which lives in this bidegree. By construction $A=B\oplus C$ as bigraded vector spaces, and it remains to show that this $C$ is really a subcomplex, i.e. that it is stable under the differential. 
From \cite{BR14} we already know that differentials of forms containing no summands with $\eta_i,\theta_i,\bar\eta_i,\bar\theta_i$-factors land again in $C$. On the other hand, whenever $\sigma$ is an elementary wedge of forms containing such a factor $\tau\in\{\eta_i,\theta_i,\bar\eta_i,\bar\theta_i\}$, writing its differential $d\sigma$ as a sum of elementary wedges we see that each of them is either a multiple of $\tau$ again, or a multiple of $dx_id\bar{x}_i$ or of $dy_{i+1}d\bar{y}_{i+1}$ or of $dz_id\bar{z}_{i+1}$ and therefore does not lie in $B$ but in $C$. Thus, the $C$ is a subcomplex and the sum $A=B\oplus C$ is one of double-complexes. As before, this implies that the Fr\"olicher spectral sequence of $B$ is a direct summand in that of $A$, hence $d_n([\beta_1])=[\beta_n]\neq 0$ on $\tilde{X}_n$.

\begin{rem}
The method we used to produce an SKT metric works in greater generality. We have not tried to find the most general statement, but the following ad-hoc construction already yields many more examples: Consider any nilpotent Lie group $G$ with left-invariant almost complex structure s.t. we have a decomposition of the space of left-invariant $(1,0)$-forms $A^1=V\oplus W$, such that $d(W)\subseteq (V+\bar V)\wedge (V+\bar V)$ and $dV=0$ ($G$ is thus $2$-step nilpotent). Assume we are given a left-invariant metric $h$ such that $\del\delbar\omega_h=\sum_{(i,j)\in A}v_{i\bar{i}j\bar{j}}$ for some basis $(v_i)_{i\in I}$ of $V$ and a subset $A\subseteq I^2$. Now construct a new Lie group $G'$, also carrying a left-invariant almost complex structure, by prescribing its space of left-invariant $(1,0)$-forms as $A^{1,0}_{G'}:=V\oplus W\oplus W'$ where $W':=\langle \omega^{ij}\rangle_{(i,j)\in A}$ has a basis element for each element in $A$ and $d\omega^{ij}:=v_{i\bar{i}}+v_{j\bar{j}}$. Then $h'$ defined by $\omega_{h'}:=\omega_h+\sum_{(i,j)\in A} \omega^{ij}\overline{\omega^{ij}}$ is a pluriclosed metric on $G'$. Note that by construction $G'$ is an extension of $G$ by an abelian subgroup. If we started with a nilmanifold $X=G/\Gamma$, we may in this way obtain a torus bundle over $X$ which carries a pluriclosed metric. 
\end{rem}

\section{Tables}

\begin{center}
\begin{table}[h]
 \centering
\begin{tabular}{>{$\mathbf\bgroup}l<{\mathbf\egroup$} || >{$}l<{$}||>{$}l<{$}}
\toprule
(p,q) & (C^{p,q}_{b})^{\sigma=\id}\text{ cases \ref{item:nakamura-1}, \ref{item:nakamura-2}}&(C^{p,q}_{b})^{\sigma=\id}\text{ case \ref{item:nakamura-3}}\\
\toprule
(0,0) & \C \left\langle 1 \right\rangle &\C \left\langle 1 \right\rangle\\
\midrule[0.02em]
(1,0) & \{0\} & \{0\}\\[5pt]
(0,1) & \{0\} & \{0\}\\
\midrule[0.02em]
(2,0) & \C \left\langle \de z_{23} \right\rangle & \C \left\langle \de z_{23} \right\rangle\\[5pt]
(1,1) & \C \left\langle \de z_{1\bar1},\; \esp^{-2z_{1}}\de z_{2\bar2}+ \esp^{2z_{1}}\de z_{3\bar3},\; \esp^{-2\bar z_{1}} \de z_{2\bar2}+ \esp^{2\bar z_{1}}\de z_{3\bar3} \right\rangle &\C \left\langle \de z_{1\bar1}\right\rangle \\[5pt]
(0,2) & \C \left\langle \de z_{\bar 2\bar3} \right\rangle & \C \left\langle \de z_{\bar 2\bar3} \right\rangle \\
\midrule[0.02em]
(3,0) & \{0\} & \{0\}\\[5pt]
(2,1) & \C \left\langle  \de z_{13\bar2}+\de z_{12\bar3},\; \esp^{-2 z_{1}}\de z_{12\bar2}+\esp^{2z_{1}}\de z_{13\bar3},\;  \esp^{-2\bar z_{1}}\de z_{12\bar2}+ \esp^{2\bar z_{1}}\de z_{13\bar3} \right\rangle &\C \left\langle \de z_{12\bar3}+\de z_{13\bar2} \right\rangle \\[5pt]
(1,2) & \C \left\langle \de z_{3\bar1\bar2}+ \de z_{2\bar1\bar3},\; \esp^{-2\bar z_{1}}\de z_{2\bar1\bar2}+\esp^{2\bar z_{1}}\de z_{3\bar1\bar3},\;   \esp^{-2 z_{1}}\de z_{2\bar1\bar2}+ \esp^{2 z_{1}}\de z_{3\bar1\bar3}\right\rangle & \C \left\langle \de z_{3\bar1\bar2}+ \de z_{2\bar1\bar3}\right\rangle\\[5pt]
(0,3) & \{0\} & \{0\} \\
\midrule[0.02em]
(3,1) & \C \left\langle \de z_{123\bar1} \right\rangle & \C \left\langle \de z_{123\bar1} \right\rangle \\[5pt]
(2,2) & \C \left\langle\de z_{23\bar2\bar3}, \; \esp^{-2z_{1}}\de z_{12\bar1\bar2},+ \esp^{2z_{1}}\de z_{13\bar1\bar3},\;  \esp^{-2\bar z_{1}}\de z_{12\bar1\bar2}+ \esp^{2\bar z_{1}}\de z_{13\bar1\bar3} \right\rangle  & \C \left\langle\de z_{23\bar2\bar3}\right\rangle\\[5pt]
(1,3) & \C \left\langle \de z_{1\bar1\bar2\bar3} \right\rangle & \C \left\langle \de z_{1\bar1\bar2\bar3} \right\rangle \\
\midrule[0.02em]
(3,2) & \{0\} & \{0\} \\[5pt]
(2,3) & \{0\} & \{0\} \\
\midrule[0.02em]
(3,3) & \C \left\langle \de z_{123\bar1\bar2\bar3} \right\rangle & \C \left\langle \de z_{123\bar1\bar2\bar3} \right\rangle \\
\bottomrule
\end{tabular}
\caption{The double complex $(C^{*,*} _{b})^{\sigma=\id}$ for the completely-solvable Nakamura manifold in cases \ref{item:nakamura-1}, \ref{item:nakamura-2} and \ref{item:nakamura-3}.}
\label{table:c-nakamura-1}
\end{table}
\end{center}

\begin{center}
\begin{table}[h]
 \centering
\begin{tabular}{>{$\mathbf\bgroup}l<{\mathbf\egroup$} || >{$}l<{$}}
\toprule
(p,q) & (A^{p,q})^{\sigma=\id} \\
\toprule
(0,0) & \C \left\langle 1 \right\rangle \\
\midrule[0.02em]
(1,0) & \{0\}  \\[5pt]
(0,1) & \{0\}\\
\midrule[0.02em]
(2,0) & \C \left\langle \omega^{12},\omega^{34} \right\rangle \\[5pt]
(1,1) & \C \left\langle \omega^{1\bar1},\omega^{2\bar2},\omega^{3\bar 3},\omega^{3\bar 4},\omega^{4\bar 3},\omega^{4\bar 4}\right\rangle \\[5pt]
(0,2) & \C \left\langle \omega^{\bar1\bar2},\omega^{\bar3\bar4}\right\rangle\\
\midrule[0.02em]
(3,0) & \{0\}\\[5pt]
(2,1) & \C \left\langle \omega^{13\bar{2}}, \omega^{14\bar{2}}, \omega^{23\bar 1}, \omega^{24\bar 1}\right\rangle\\[5pt]
(1,2) & \C \left\langle \omega^{2\bar1\bar3}, \omega^{2\bar1\bar 4}, \omega^{1\bar2\bar3}, \omega^{1\bar2\bar4}\right\rangle\\[5pt]
(0,3) & \{0\} \\
\midrule[0.02em]
(4,0) & \C \left\langle \omega^{1234}\right\rangle\\[5pt]
(3,1) & \C \left\langle \omega^{123\bar 3}, \omega^{123\bar 4},  \omega^{124\bar 3}, \omega^{124\bar 4}, \omega^{134\bar 1}, \omega^{234\bar 2}\right\rangle\\[5pt]
(2,2) & \C \left\langle \omega^{12\bar1\bar 2}, \omega^{12\bar3\bar 4}, \omega^{13\bar1\bar 3},\omega^{13\bar 1\bar 4}, \omega^{14\bar1\bar3}, \omega^{14\bar1\bar4}, \omega^{23\bar 2\bar 3}, \omega^{23\bar 2\bar 4}, \omega^{24\bar2\bar 3},\omega^{24\bar2\bar 4}, \omega^{34\bar1\bar2}, \omega^{34\bar3\bar4}\right\rangle\\[5pt]
(1,3) & \C \left\langle \omega^{3\bar1\bar2\bar 3}, \omega^{4\bar1\bar2\bar 3}, \omega^{3\bar1\bar2\bar4}, \omega^{4\bar1\bar2\bar4}, \omega^{1\bar 1\bar 3\bar 4},\omega^{2\bar2\bar3\bar 4}\right\rangle \\[5pt]
(0,4) & \C \left\langle \omega^{\bar1\bar2\bar3\bar4}\right\rangle\\
\bottomrule
\end{tabular}
\caption{The double complex $A^{\sigma=\id}$ for Construction \ref{constr: ddc+3}, up to middle degree.}
\label{table: ddc+3}
\end{table}
\end{center}

\end{document}